\documentclass{mfatshort}
\newtheorem{theorem}{Theorem}

\newtheorem{corollary}{Corollary}
\theoremstyle{definition}
\newtheorem{definition}{Definition}

\newtheorem{remark}{Remark}
\newtheorem{proposition}{Proposition}

\def\mul{{\rm mul\,}}
\newcommand\dom{\operatorname{dom}}
 
\def\Ext{{\rm Ext\,}}
\def\op{{\rm op\,}}
\newcommand{\ran}{{\mathrm{ran\,}}}
\newcommand{\gG}{{\Gamma}}
\newcommand{\cH}{{\mathcal H}}
\newcommand{\cC}{{\mathcal C}}
\newcommand{\gotH}{{\mathfrak H}}

\begin{document}

\title
    {On the extremal  extensions of a non-negative    Jacobi operator}

\author{Aleksandra Ananieva}
\address{R. Luxemburg str. 74, \, 83114, Donetsk, Ukraine}
\curraddr{}
\email{ananeva89@gmail.com}
\thanks{The authors  express their gratitude to  Prof.  M. Malamud for posing the problem  and permanent  attention to the work and to  Prof. L. Oridoroga  for useful discussions.
The second author  was supported by  FAPESP grant 2012/50503-6.}

\author{Nataly Goloshchapova }
\address{ Rua do Mat\~ao, 1010, \, 05508-090, S\~ao Paulo,  Brazil}
\email{nataliia@ime.usp.br}


\subjclass[2000]{47B36; 47B36 }
\date{}
\dedicatory{Dedicated with deep respect to Professor Ya.V. Mykytyuk on
the occasion of his 60th birthday.} \keywords{Jacobi matrix, non-negative operator,
Friedrichs and Krein  extensions, boundary triplet, Weyl function}

\begin{abstract}
We consider minimal non-negative Jacobi operator with $p\times p-$matrix
entries. Using the technique  of  boundary  triplets and the
corresponding Weyl functions, we describe
the  Friedrichs and Krein extensions of the minimal Jacobi operator.
Moreover, we  parameterize the set of  all non-negative self-adjoint extensions  in
terms  of  boundary conditions.
  \end{abstract}

\maketitle

\section{Introduction}

Let  $A$ be a densely defined non-negative symmetric operator in  the Hilbert space
$\mathfrak{H}$. Since $A$ is non-negative, then by Friedrichs-Krein theorem, it admits
non-negative self-adjoint extensions. Qualified description of all non-negative
self-adjoint extensions of $A$ and also criterion of uniqueness of non-negative
self-adjoint extension of $A$  were first given by Krein in \cite{Kre47}. His results
were generalized in numerous papers (see \cite{ArlTse05, DerMal91, DerMal95} and
references therein).

 Among  all non-negative self-adjoint extensions of $A$, two (extremal) extensions   are particularly interesting and
 important
enough to have a name. The Friedrichs extension (so-called "hard" extension) $A_F$ is
the "greatest" one in the sense of quadratic forms. It is given by  restriction of $A^*$
to the domain
\[
\dom(A_F)=\left\{ \begin{array}{c}
                     u\in \dom(A^*):\exists\,\, u_k\in \dom(A) \;\text{such  that}\;
\|u-u_k\|_\mathfrak{H}\rightarrow 0 \; \\ \text{as}\; k\rightarrow \infty\;
 \mbox{and }\; (A(u_j-u_k),u_j-u_k)_\mathfrak{H} \rightarrow 0 \;\text{as}\;
j, k\rightarrow \infty
                  \end{array} \right\}.
\]
In other words, $A_F$  is the  self-adjoint operator associated with the closure of the
symmetric form
$$
\mathfrak{t}[u,v]=(Au,v)_\mathfrak{H}, \qquad u,v\in \dom(A).
$$
The  Krein extension ("soft" extension) $A_K$ is defined to be restriction of $A^*$ to
the domain
\begin{equation}\label{domA_K}\dom(A_K)=\left\{ \begin{array}{c}
                     u\in \dom(A^*):\exists\,\, u_k\in \dom(A) \;\text{such  that}\;
\|A^*u-Au_k\|_\mathfrak{H}\rightarrow 0 \; \\ \text{as}\; k\rightarrow \infty\;
 \mbox{and }\; (u_j-u_k,A(u_j-u_k))_\mathfrak{H} \rightarrow 0 \;\text{as}\;
j, k\rightarrow \infty
                  \end{array} \right\}.\end{equation}
If $A$ is positive definite, $A\geq \varepsilon I>0,$ then ~\eqref{domA_K} takes the form
\begin{equation}\label{A_K}
\dom(A_K)=\dom(\overline{A}) \dot{+}\ker(A^*).
\end{equation}
Krein proved in \cite{Kre47} that all the non-negative self-adjoint extensions
$\widetilde{A}$ of $A$ lie between $A_F$ and $A_K$, i.e.,
\[((A_F+aI)^{-1}u,u)_\mathfrak{H}\leq((\widetilde{A}+aI)^{-1}u,u)_\mathfrak{H}\leq
((A_K+aI)^{-1}u,u)_\mathfrak{H},\quad u\in\mathfrak{H},\] for  any $a>0$.

  In the present  paper we are dealing with the problem of description of the extremal extensions of non-negative  Jacobi
  operator to be defined below.
 Let $A_j, B_j\in \mathbb{C}^{p\times p}$. Moreover, we assume that matrices  $A_j$ are
self-adjoint  and  matrices  $B_j$  are   invertible for each $j\geq
0$ (see \cite[Chapter VII, \S 2]{Ber68}). We consider  semi-infinite Jacobi  matrix  with  matrix entries
\begin{equation*}
\mathbf{J}= \left(
  \begin{array}{cccc}
    A_0 & B_0 & O_p & \ldots \\
    B^*_0 & A_1 & B_1 & \ldots \\
    O_p &  B_1^*& A_2 & \ldots \\
    \vdots & \vdots & \vdots & \ddots \\
  \end{array}
\right),
\end{equation*}
where $O_p$ is zero $p\times p$ matrix. Given a sequence $u=(u_j),\,\,
u_j\in\mathbb{C}^p$, $\mathbf{J}u$  is  again  a sequence of column  vectors. If  we set
$B_{-1}=O_p$,
\[(\mathbf{J}u)_j=B_{j}u_{j+1}+A_ju_j+B_{j-1}^*u_{j-1},\quad j\geq 0.\]
The  maximal  operator $T_{\max}$ is defined  by
\begin{equation*}\label{max}
(T_{\max}u)_j=(\mathbf{J}u)_j,\quad j\geq 0
\end{equation*}
 on  the  domain
\[\dom(T_{\max})=\{u\in l^2_p:\,\, \mathbf{J}u\in l^2_p\}.\]

The minimal operator $T_{\min}$  is the  closure in $l^2_p$ of the  preminimal
operator $T$ which is the  restriction of $T_{\max}$ to the domain
\begin{equation*}\label{min}
\dom(T)=\{u\in l^2_p:\,\,u_j=0 \text{\,\,for  all but a finite  number of values of }
j\}.
\end{equation*}
 It is straightforward to see that $T_{\min}$ is a densely defined symmetric
operator and
\[ T_{\min}^*=T_{\max},\quad T_{\max}^*=\overline{T}=T_{\min}. \]
Deficiency indices $n_\pm(T_{\min})=\dim(\ker(T_{\max}\pm zI)),\,\,z\in\mathbb{C}_+,$
satisfy the inequality $0\leq n_+(T_{\min}), n_-(T_{\min}) \leq p$ (see \cite[Chapter
VII, \S 2]{Ber68}). In  the following,  we shall assume that  $n_\pm(T_{\min})=p$
(completely indefinite case takes place)  and  $T$ is non-negative. Note that
non-negativity of $T$ implies non-negativity of $A_j,\; j\geq 0$.


Examples of symmetric block Jacobi matrices generating symmetric operators with
arbitrary possible values of the deficiency numbers were constructed in \cite{Dyu2010}.

The  problem of description of  the extremal  extensions   $T_F$  and $T_K$  in the
scalar  case $(p=1)$ was studied  in the number of papers. Description of the Friedrichs
domain in terms of a weighted Dirichlet sum was  obtained in \cite{Ben94}.

 In \cite{Sim98}, Simon showed that certain matrix operators that approximate the
Friedrichs and Krein extensions converge in the strong resolvent norm. In \cite{BrChr05}
Brown and Christiansen (assuming that $T$ is \emph{positive  definite}) obtained  the description of  $T_F$  and $T_K$  using  the concept of the so-called minimal solution (see also \cite{MarZet00}).

 The purpose of this work is to generalize  at least partially the  results obtained  in \cite{BrChr05}   to the case of arbitrary $p\in\mathbb{N}$ and \emph{non-negative}  operator $T$.  We use  an abstract
            description of extremal non-negative extensions  obtained by V.~Derkach and M.~Malamud in the  framework   of
            boundary  triplets  and the  corresponding Weyl  functions  approach (see  \cite{DerMal91, GG} and also  Section 2 for the precise  definitions). In
            particular, we show that mentioned  results  from  \cite{BrChr05} might be expressed in terms of boundary
            triplets theory.\\
                        \\

            \textbf{Notation.}  In what follows $\mathbb{C}^{p\times p}$  denotes the  set of
$p\times p$ complex-valued matrices;   $l^2_p$ denotes Hilbert space of infinite
sequences $u=(u_j),\,\, u_j\in\mathbb{C}^p$ equipped with   inner product
$(u,v)_{l^2_p}=\sum\limits_{j=0}^\infty v^*_ju_j$. The set of closed (bounded) operators
in the Hilbert space $\cH$ is denoted by $\cC(\cH)\,$ (respectively,
$\mathcal{B}(\cH)$).

 \section{Linear relations, boundary  triplets  and Weyl  functions}
                Let $A$ be a closed densely defined symmetric operator in
the separable  Hilbert space $\mathfrak H $ with equal deficiency indices
$n_\pm(A)=\dim\ker(A^*\pm iI)\leq\infty$.

\begin{definition} \cite{GG} %
A triplet $\Pi=\{\cH,\gG_0,\gG_1\}$ is called a \textit{boundary triplet} for the
adjoint operator $A^*$ of $A$ if $\cH$ is an auxiliary Hilbert space and
$\Gamma_0,\Gamma_1:\  \dom(A^*)\rightarrow\cH$ are linear mappings such that
\begin{itemize}
\item [\textit{(i)}] the  second  Green identity,
\begin{equation*}\label{GI}
(A^*f,g)_\mathfrak H  - (f,A^*g)_\mathfrak H  = (\gG_1f,\gG_0g)_\cH -
(\gG_0f,\gG_1g)_\cH,
\end{equation*}
holds for all $f,g\in\dom(A^*)$, and
\item [\textit{(ii)}] the mapping $\gG:=(\Gamma_0,\Gamma_1)^\top: \dom(A^*) \rightarrow \cH \oplus\cH$
is surjective.
\end{itemize}
\end{definition}
With each boundary triplet $\Pi=\{\cH,\gG_0,\gG_1\}$  one associates two
self-adjoint extensions $A_j:=A^*\upharpoonright \ker (\Gamma_j), \; j\in\{0,1\}.$
\begin{definition}
\
\begin{itemize}
\item[\textit{(i)}]  A  \textit{closed  linear relation} $\Theta$ in  $\mathcal{H}$ is a closed subspace of
$\mathcal{H}\oplus\mathcal{H}$.   The \textit{domain}, the \textit{range}, and the
\textit{multivalued part} of $\Theta$ are defined as follows
\begin{gather*}\dom(\Theta) :=\bigl\{
f:\{f,f'\}\in\Theta\bigr\},\quad \ran(\Theta): =\bigl\{
f^\prime:\{f,f'\}\in\Theta\bigr\},\\ \mul(\Theta): =\bigl\{
f^\prime:\{0,f'\}\in\Theta\bigr\}.\end{gather*}
\item[\textit{(ii)}] A linear  relation  $\Theta$  is \textit{symmetric}
 if  $$(f',h)_{\cH}-(f,h')_{\cH}=0\quad  \mbox{for all}\quad  \{f,f'\},  \{h,h'\}\in  \Theta.$$
 \item[\textit{(iii)}] The \textit{adjoint} relation
$\Theta^*$  is defined by
\begin{equation*}
\Theta^*= \big\{ \{h,h^\prime\}: (f^\prime,h)_{\cH}=(f,h^\prime)_{\cH}\ \ \text{for
all}\ \ \{f,f^\prime\} \in\Theta\big\}.
\end{equation*}
\item[\textit{(iv)}]  A closed linear   relation $\Theta$  is  called  \textit{self-adjoint}
 if both $\Theta$ and $\Theta^*$ are  maximal symmetric, i.e., they do not admit symmetric extensions.
 \end{itemize}
    \end{definition}
     For the
symmetric relation $\Theta\subseteq\Theta^*$ in $\cH$ the multivalued part
$\mul(\Theta)$ is  orthogonal to $\dom(\Theta)$ in $\cH$. Setting $\cH_{\rm
op}:=\overline{\dom(\Theta)}$ and $\cH_\infty:=\mul(\Theta)$, one verifies that $\Theta$
can be rewritten as the direct orthogonal sum of a self-adjoint operator $\Theta_{\rm
op}$ (\emph{operator part} of $\Theta$) in the subspace $\cH_{\rm op}$ and a ``pure''
relation $\Theta_\infty=\bigl\{\{0,f'\}:f'\in\mul(\Theta)\bigr\}$ in the subspace
$\cH_\infty$.

\begin{proposition}\cite{DerMal91,GG}\label{propo}
Let  $\Pi=\{\mathcal{H},\Gamma_0,\Gamma_1\}$  be   a boundary triplet for $A^*$. Then
the  mapping
\begin{equation}\label{ext}
\Ext_A\ni\widetilde{A}:=A_\Theta\rightarrow
\Theta:=\Gamma(\dom(\widetilde{A}))=\big\{\{\Gamma_0f,\Gamma_1f\}:\,\,f\in
\dom(\widetilde{A})\big\}
\end{equation}
establishes  a  bijective  correspondence   between  the  set  of all  closed proper
extensions $\Ext_A$  of $A$ and  the set of all closed   linear  relations
$\widetilde{\mathcal{C}}(\mathcal{H})$ in $\mathcal{H}$. Furthermore,  the following
assertions  hold.
\begin{itemize}
\item[\textit{(i)}]
  The  equality $(A_\Theta)^*=A_{\Theta^*}$
%
%
holds  for any  $\Theta\in\widetilde{\mathcal{C}}(\mathcal{H})$.

\item[\textit{(ii)}]
 The extension $A_\Theta$  in \eqref{ext} is  symmetric
(self-adjoint)  if  and only  if  $\Theta$  is symmetric (self-adjoint).
\item[\textit{(iii)}]
  If,  in  addition,    extensions $A_\Theta$ and
$A_0$ are disjoint, i.e.,
$\dom(A_\Theta)\cap\dom(A_0)=\dom(A)$, then \eqref{ext} takes the form
\begin{equation*}
A_\Theta=A_B=A^*\!\upharpoonright\ker\bigl(\Gamma_1-B\Gamma_0\bigr),\quad
B\in\mathcal{C}(\mathcal{H}).
   \end{equation*}
\end{itemize}
 \end{proposition}
\begin{remark} 
In the case  $n_{\pm}(A)=n<\infty$,
 any proper  extension $A_{\Theta}$ of the operator $A$ admits representation (see \cite{DerMal00})
\begin{equation}\label{cd}
A_{\Theta}:=A_{C,D}=A^*\upharpoonright\ker\bigl(D\Gamma_1-C\Gamma_0\bigr),\quad C,D \in \mathcal{B}(\mathcal{H}).
\end{equation}
Moreover, according to the Rofe-Beketov theorem \cite{R-B69} (see also  \cite[Theorem 125.4]{Akh78}), $A_{C,D}$
 is self-adjoint if and only if
 $C, D$, satisfy the following  conditions
\begin{equation}\label{sacond}
CD^*=DC^* \quad\text{and}\quad 0\in\rho(CC^*+DD^*).
\end{equation}
\end{remark}
\begin{definition}\cite{DerMal91}\label{Weylfunc}
Let $\Pi=\{\cH,\gG_0,\gG_1\}$ be a boundary triplet  for $A^*$. The operator-valued
function   $M(\cdot) :\rho(A_0)\rightarrow \mathcal{B}(\cH)$ defined by
\begin{equation*}\label{2.3A}
\Gamma_1f_z=M(z)\Gamma_0f_z,\quad\text{ for all } f_z\in\ker(A^*-zI),\qquad
z\in\rho(A_0),
      \end{equation*}
      is  called  \textit{the Weyl function}, corresponding to the triplet $\Pi.$
      \end{definition}
\begin{proposition}\cite{DerMal91, DerMal95}\label{prkf}
Let $A$ be  a densely defined nonnegative symmetric operator with finite deficiency
indices  in $\gotH$, and  let $\Pi=\{\cH,\gG_0,\gG_1\}$ be a boundary
 triplet for  $A^*$.
\, Let  also  $M(\cdot)$ be the corresponding Weyl function.  Then the  following assertions
hold.
\begin{itemize}
\item[\textit{(i)}]  There exist  strong   resolvent  limits
\begin{equation}\label{W_lim}
M(0):=s-R-\lim\limits_{x\uparrow 0}M(x), \qquad M(-\infty):=s-R-\lim\limits_{x\downarrow
-\infty}M(x).
\end{equation}
\item[\textit{(ii)}]  $M(0)$   and $M(-\infty)$ are
 self-adjoint linear relations in $\cH$ associated with  the
 semibounded  below (above) quadratic  forms
 $$\mathrm{t}_0[f]=\lim\limits_{x\uparrow 0}(M(x)f,f)\geq \beta||f||^2, \quad \mathrm{t}_{-\infty}[f]=\lim\limits_{x\downarrow
 -\infty}(M(x)f,f)\leq \alpha ||f||^2,$$  and
 \begin{align*}\label{Weylform}
 \dom(\mathrm{t}_0)&=\big\{f\in \cH:\,\lim\limits_{x\uparrow 0}|(M(x)f,f)|<\infty\big\}=\dom((M(0)_{\op}-\beta)^{1/2}),\\
\dom(\mathrm{t}_{-\infty})&=\big\{f\in \cH:\,\lim\limits_{x\downarrow
 -\infty}|(M(x)f,f)|<\infty\big\}=\dom((\alpha-M(-\infty)_{\op})^{1/2}).
 \end{align*}
 Moreover,
 \begin{align*}
\dom(A_K)&=\{f\in\dom(A^*):\,\,\{\Gamma_0f,\Gamma_1f\}\in M(0)\},\\
\dom(A_F)&=\{f\in\dom(A^*):\,\,\{\Gamma_0f,\Gamma_1f\}\in M(-\infty)\}.
 \end{align*}
 \item[\textit{(iii)}] Extensions
$A_0$  and $A_K$  are  disjoint  ($A_0$ and $A_F$ are disjoint) if and only if
$$M(0)\in\cC(\cH) \qquad (M(-\infty)\in\cC(\cH),\ \  \mbox{respectively}).$$
Moreover,
\begin{align*}
\dom(A_K)&=\dom(A^*)\upharpoonright\ker(\Gamma_1-M(0)\Gamma_0)\\
(\dom(A_F)&=\dom(A^*)\upharpoonright\ker(\Gamma_1-M(-\infty)\Gamma_0),\ \
\mbox{respectively}).
\end{align*}
\item[\textit{(iv)}]  $A_0=A_K$   ($A_0=A_F$) if  and only  if
 \begin{align*} 
 \lim_{x\uparrow0}(M(x)f,f)&=+\infty,\quad f\in\cH\setminus\{0\} \\
 (\lim_{x\downarrow-\infty}(M(x)f,f)&=-\infty, \quad f\in\cH\setminus\{0\}, \quad \mbox{respectively}).
 \end{align*}
 \item[\textit{(v)}]   If, in addition, $A_0=A_F$  ($A_0=A_K$), then  the set of all non-negative self-adjoint extensions of $A$
 admits
 parametrization
 \eqref{ext},  where $\Theta$  satisfies
\begin{equation}\label{kappa1}
 \Theta-M(0)\geq 0\qquad  (\Theta-M(-\infty)\leq 0,\,\, \mbox{respectively}).
 \end{equation}
  Moreover, if \eqref{kappa1} does not hold, the number of negative eigenvalues of arbitrary self-adjoint extension $\kappa_-(A_\Theta)$ is given by
  \begin{equation*}\label{kappa}
 \kappa_-(A_\Theta)=\kappa_-(\Theta-M(0))\quad  (\kappa_-(A_\Theta)=\kappa_-(M(-\infty)-\Theta),\,\, \mbox{respectively}).
 \end{equation*}
 \end{itemize}
\end{proposition}
\begin{remark}
We should mention that   the existence of the  limits in \eqref{W_lim} follows from
finiteness of deficiency indices of
 $A$.
\end{remark}
\begin{remark}
Note that if the lover bound of $A$ is zero and the spectrum of $A_F$ is purely  discrete, then $A_F$ and $A_K$
are not disjoint. In this case $M(0)$ is a linear relation if $A_0\geq 0.$
\end{remark}

  \begin{corollary}\label{negind}
 Let assumptions of  Proposition \ref{prkf} hold and $A_0=A_K$. Let also $A_{C,D}$  be self-adjoint extension of $A$ defined by \eqref{cd}. Then $A_{C,D}$ is non-negative if and only if
 \begin{equation}\label{poscond}CD^*-DM(-\infty)D^*\leq 0.\end{equation} Moreover, if \eqref{poscond} does not hold, the number of negative eigenvalues of $A_{C,D}$ (counting multiplicities) coincides  with the number of positive eigenvalues of the linear relation
 $CD^*-DM(-\infty)D^*$, i.e.,
\begin{equation*}\label{indic}
\kappa_-(A_{C,D})=\kappa_+(CD^*-DM(-\infty)D^*).
\end{equation*}
\end{corollary}

\begin{corollary}\label{krein}\cite{DerMal95}
Suppose that $A_F$ has  purely  discrete spectrum. Then the Krein
extension $A_K$ is given by
 \begin{equation*}
 \dom(A_K)=\dom(\overline{A})\dotplus \ker(A^*).
 \end{equation*}
 Moreover, the spectrum of  $A_K\upharpoonright \ker(A^*)^{\bot}$ is purely discrete.
\end{corollary}

\section{Extremal extensions of $T_{\min}$}
As usual, we  denote by $(P_j(z))$ the  solution  to the  matrix equation
\[(\mathbf{J}U)_j=zU_j,\quad j\geq 0\]
with the initial conditions  $P_0(z)=I_p,\,\,P_1(z)=B_0^{-1}(zI_p-A_0).$  Here $I_p\in
\mathbb{C}^{p\times p}$   is the identity  matrix. Furthermore, we denote by $(Q_j(z))$
the solution to
\[(\mathbf{J}U)_j=zU_j,\,\,j\geq 1\]  with  $Q_0(z)=O_p$ and $Q_1(z)=B_0^{-1}$.  The  matrix functions $P_j(z)$ and $Q_j(z)$
are  polynomials in the  complex variable $z$ of degree $j$  and $j-1$, respectively,
with  matrix coefficients. We mention that   $(P_j(z))$  and  $(Q_j(z))$   are called
matrix polynomials of the first  and   second kind, respectively.

 Following   \cite{DerMal92}, we  define a   boundary  triplet  $\Pi=\{\cH,\Gamma_0,\Gamma_1\}$ for $T_{\max}$ by setting
\begin{equation}\label{tripl}
\mathcal{H}=\mathbb{C}^p,\qquad \Gamma_1u=(Q(0))^* T_{\max}u-P_0u,\qquad
\Gamma_0u=(P(0))^* T_{\max}u,
\end{equation}
where $u\in\dom(T_{\max})$
 and  $P_0$  is  orthoprojection  in
$\bigoplus\limits_{j=0}^\infty\mathcal{H}_j$ onto
$\mathcal{H}_0$, in which $\mathcal{H}_j=\mathbb{C}^p$.

We  should  note  that the mappings  $(P(0))^*$  and $(Q(0))^*$  act as infinite
$"p\times \infty"$ matrices, i.e., $(P(0))^*, (Q(0))^*:\,l^2_p\rightarrow \mathbb{C}^p$.
Each their row is  "constructed"  by the corresponding rows of $P_j^*(0)$  and
$Q_j^*(0)$, respectively.

It is easily seen that
\begin{equation*}
\Gamma_1(P_j(z))=z\sum\limits_{j=1}^{\infty}Q_j^*(0)P_j(z)-I_p,\qquad
\Gamma_0(P_j(z))=z\sum\limits_{j=0}^{\infty}P_j^*(0)P_j(z),
\end{equation*}
and, by  Definition \ref{Weylfunc}, we get
 \begin{multline*}
\qquad\qquad\qquad\qquad M(z)=\Gamma _1(P_j(z))(\Gamma
_0(P_j(z)))^{-1}\\=\left(z\sum\limits_{j=1}^{\infty}Q_j^*(0)P_j(z)-I_p\right)\cdot
\left(z\sum\limits_{j=0}^{\infty}P_j^*(0)P_j(z)\right)^{-1}.\qquad
\end{multline*}
Applying Proposition \ref{prkf}  to the  operator $T_{\min}$, we obtain  the  following result.
\begin{theorem}\label{kre}
 Let  $\Pi=\{\mathcal{H},\Gamma_0,\Gamma_1\}$  be the  boundary triplet for $T_{\max}$  given by \eqref{tripl} and let $M(\cdot)$ be the corresponding Weyl function.
  Then the following assertions hold.
\begin{itemize}
\item[\textit{(i)}]
The Krein extension $T_K$  coincides with  $T_0$, i.e.,
  \[\dom(T_K)=\dom(T_{\max})\upharpoonright \ker(\Gamma_0)=\{u\in T_{\max}:\,(P(0))^*T_{\max}u=0\}.\]
\item[\textit{(ii)}] Self-adjoint extension  $T_{\Theta}$ is non-negative if and only
if it admits representation
\begin{equation*}
T_{\Theta}=T_{C,D}=T_{\max}\upharpoonright\ker(D\Gamma_1-C\Gamma_0),
\end{equation*}
 where $C, D$ satisfy \eqref{sacond} and \eqref{poscond}.
 \item[\textit{(iii)}] The number of negative eigenvalues of  $T_{C,D}=T_{C,D}^*$ (defined by \eqref{cd}) coincides with the number of positive eigenvalues of the linear relation
 $CD^*-DM(-\infty)D^*$, i.e.,
\begin{equation*}
\kappa_-(T_{C,D})=\kappa_+(CD^*-DM(-\infty)D^*).
\end{equation*}
\end{itemize}
\end{theorem}
\begin{proof}
$(ii)$
The statement might be proven at least in two  different ways.

 ${\bf 1.}$ Since $M(x)$ is holomorphic  and increasing in $(-\varepsilon,0)$  (see \cite{DerMal95}), the strong  limit
 $s-\lim\limits_{x\uparrow 0}(M(x)+\gamma)^{-1}$  exists for any  $\gamma>0$. Namely,

\begin{multline}\label{M} M_{\gamma}(0):=s-\lim\limits_{x\uparrow
0}\big(M(x)+\gamma\big)^{-1}=s-\lim\limits_{x\uparrow
0}\left(x\sum\limits_{j=0}^{\infty}P_j^*(0)P_j(x)\right)\\
\times\left(x\sum\limits_{j=1}^{\infty}Q_j^*(0)P_j(x )-I_p+\gamma
x\sum\limits_{j=0}^{\infty}P_j^*(0)P_j(x)\right)^{-1}=O_p.
\end{multline}
Indeed, since  $n_\pm(T_{\min})=p$ (see \cite{KosMir98}), the series
$\sum\limits_{j=0}^{\infty}||P_j(z)||^2$  and $\sum\limits_{j=0}^{\infty}||Q_j(z)||^2$
converge uniformly on each bounded subset $\mathbb{C}$ (see, for instance, \cite[Theorem
1]{KosMir98}).  Therefore, we can pass to the limit  in  \eqref{M} under the sum sign
as $x\rightarrow 0$.
 Hence $M_{\gamma}^{-1}(0)=\{0,\cH\}=\{ \{0,f\}: f\in
 \cH\}$.  Since
 $$\dom(T_0)=\{u\in \dom(T_{\max}):\, \Gamma_0u=0\}=\{u\in \dom(T_{\max}): \{\Gamma_0u,\Gamma_1u\}\in\{0,\cH\}\},$$
 by Proposition \ref{prkf} $(ii)$, we arrive at  $T_K=T_0$.
 \\
 ${\bf 2.}$  Suppose, in addition, that $T\geq \varepsilon I>0$. Using the equality  $\dom(T_0)=\dom(T_{\max})\upharpoonright\ker(\Gamma_0)$,
we easily get  from \eqref{tripl} that $\ker(T_{\max})\subset \dom(T_0)$.  Therefore,
\eqref{A_K} implies  the inclusion $\dom(T_0)\supset \dom(T_K)$. Since  $T_0$ and $T_K$
are self-adjoint,  we arrive at  $T_K=T_0$.

$(ii)$ and $(iii)$  easily follow from Corollary \ref{negind}  and assertion $(i)$.
\end{proof}

Assume   now  that    $p=1$  and $A_j, B_j$ are positive real numbers. It is known that
$T_{\min}$ is connected  with some Stiltjes moment problem, see \cite{Akh65}. Briefly,
a Stiltjes moment problem has a following description. Given a sequence $\gamma_0,
\gamma_1, \gamma_2,... $ of reals. When is there a measure, $d\mu$  on $[0,\infty)$ so
that
\[\gamma_n=\int\limits_0^\infty x^n d\mu(x)\]
and if such a $\mu$ exists, is it unique?

The operator  $T_{\min}$  is self-adjoint  if  and only if associated Stiltjes  moment
problem is  determinate, i.e.,  it  has unique solution.
   Since  $n_\pm(T_{\min})=1$, the  determinacy  does not take place and, therefore,    sequence  $\frac{Q_j(0)}{P_j(0)}$
   converges  (see
\cite[Theorem 0.4, p. 293]{Akh65} or \cite[Section 3]{Berg94}).

The  existence  of this limit  is  the key fact for the description of the  Friedrichs
extension which  we are going  to present.
  In particular, the limit
\begin{equation}\label{a}
  \alpha:=\lim\limits_{j\rightarrow \infty}\frac{Q_j(0)}{P_j(0)}
\end{equation} is negative.  Indeed,  since all the zeros of the  polynomials $P_j(x)$ and $Q_j(x)$  lie in the
interval $[0,\infty)$ (see, \cite[Chapter I]{Chi_78}), $P_j(\cdot)$  and $Q_j(\cdot)$
do not change the sign  in $(-\infty, 0)$. Noticing that   $P_j(x)/Q_j(x)<0$   for $x<0$
large enough (the degrees  of  $P_j(\cdot)$  and $Q_j(\cdot)$ are $j$ and $j-1$,
respectively, and leading coefficients  equal $B_{j-1}^{-1}\cdot..\cdot B_0^{-1}>0$), we
get the negativity of $\alpha$.
\begin{theorem}\label{T_F}
Assume $p=1$.   Let also $\Pi=\{\mathcal{H},\Gamma_0,\Gamma_1\}$  be the  boundary
triplet for $T_{\max}$  defined by \eqref{tripl} and $M(\cdot)$ be the corresponding
Weyl function. The domain of the Friedrichs extension is given by
 \begin{equation}\label{fr}
 \dom(T_F)=\{u\in \dom(T_{\max}):\, (\Gamma_1-\alpha \Gamma_0)u=0\},
 \end{equation} 
where $\alpha$  is  defined by \eqref{a}.
\end{theorem}

\begin{proof}
To prove  the  statement we    use Proposition \ref{prkf}$(iii)$.  Namely,  it is sufficient to show   that
\begin{equation}\label{alpha} M(-\infty)=\lim\limits_{x\downarrow -\infty}M(x)=\lim\limits_{x\downarrow
-\infty}\frac{x\sum\limits_{j=1}^{\infty}Q_j(0)P_j(x)-1}{x\sum\limits_{j=0}^{\infty}P_j(0)P_j(x)}=\alpha.
\end{equation}

  Since  the orthogonal polynomials do not change the sign  in $(-\infty, 0)$,
\begin{equation}\label{polyn}
 P_j(0)P_j(x)=|P_j(0)P_j(x)| \qquad\text{and}\qquad Q_j(0)P_j(x)=-|Q_j(0)P_j(x)|.
\end{equation}
Thus, the sequence
$\sum\limits_{j=0}^{n}P_j(0)P_j(x)=\sum\limits_{j=0}^{n}|P_j(0)P_j(x)|$  increases as
$n\rightarrow\infty$ and, therefore,
$$\lim\limits_{x\downarrow -\infty}M(x)=\lim\limits_{x\downarrow
-\infty}\frac{x\sum\limits_{j=1}^{\infty}Q_j(0)P_j(x)-1}{x\sum\limits_{j=0}^{\infty}P_j(0)P_j(x)}=-\lim\limits_{x\downarrow
-\infty}\frac{\sum\limits_{j=1}^{\infty}|Q_j(0)P_j(x)|}{\sum\limits_{j=0}^{\infty}|P_j(0)P_j(x)|}.$$

It follows from \eqref{a} that for  any small   $\delta>0$ there exists $N=N(\delta)$
such that estimate
$$\alpha-\delta<\frac{Q_j(0)}{P_j(0)}<\alpha+\delta$$ holds for  $j>N(\delta)$.
Combining this inequality with \eqref{polyn}, we get
$$\alpha-\delta< \frac{Q_j(0)P_j(x)}{|P_j(0)P_j(x)|}< \alpha+\delta, \qquad x\in(-\infty,0), \quad j\geq N(\delta).$$
 The latter is equivalent to
 \begin{equation*}\label{epsilon}
 (\alpha-\delta)|P_j(0)P_j(x)|<Q_j(0)P_j(x)<(\alpha+\delta)|P_j(0)P_j(x)|, \quad x\in(-\infty,0), \; j\geq
 N(\delta).
 \end{equation*} Hence
\begin{equation}\label{bi}
\begin{split}
(\alpha-\delta)\frac{\sum\limits_{j=N(\delta)}^{\infty}|P_j(0)P_j(x)|}{\sum\limits_{j=0}^{\infty}|P_j(0)P_j(x)|}
&<\frac{\sum\limits_{j=1}^{\infty}Q_j(0)P_j(x)}{\sum\limits_{j=0}^{\infty}|P_j(0)P_j(x)|}-\frac{\sum\limits_{j=1}^{N(\delta)-1}Q_j(0)P_j(x)}{\sum\limits_{j=0}^{\infty}|P_j(0)P_j(x)|}\\
&<(\alpha+\delta)\frac{\sum\limits_{j=N(\delta)}^{\infty}|P_j(0)P_j(x)|}{\sum\limits_{j=0}^{\infty}|P_j(0)P_j(x)|},
\qquad x\in(-\infty,0),\ \ j\geq
 N(\delta).
\end{split}
\end{equation}
 Since
$$\sum_{j=0}^{\infty}P_j(x)P_j(0)=\sum_{j=0}^{\infty}|P_j(x)P_j(0)|>|P_{N(\delta)}(x)P_{N(\delta)}(0)|,
$$
and $P_j$ is polynomial of degree $j$, we get
\begin{equation}\label{sum*}
0\leq\lim\limits_{x\rightarrow
-\infty}\frac{\sum\limits_{j=1}^{N(\delta)-1}|P_j(0)P_j(x)|}{\sum\limits_{j=0}^{\infty}|P_j(0)P_j(x)|}<\lim\limits_{x\rightarrow
-\infty}\frac{\sum\limits_{j=1}^{N(\delta)-1}|P_j(0)P_j(x)|}{|P_{N(\delta)}(x)P_{N(\delta)}(0)|}=0.
\end{equation}
Similarly we obtain
\begin{equation}\label{sum**}
\lim\limits_{x\rightarrow
-\infty}\frac{\sum\limits_{j=1}^{N(\delta)-1}Q_j(0)P_j(x)}{\sum\limits_{j=0}^{\infty}|P_j(0)P_j(x)|}=0.
\end{equation}
Taking into account that
\begin{equation*}\label{sum}
\sum\limits_{j=0}^{\infty}|P_j(0)P_j(x)|=\sum\limits_{j=N(\delta)}^{\infty}|P_j(0)P_j(x)|+\sum\limits_{j=0}^{N(\delta)-1}|P_j(0)P_j(x)|,
\end{equation*}
we get from \eqref{bi}--\eqref{sum**} the following inequality
\begin{equation*}\label{inequality}
\alpha-\delta\leq \lim\limits_{x\rightarrow
-\infty}\frac{\sum_{j=1}^{\infty}Q_j(0)P_j(x)}{\sum_{j=0}^{\infty}|P_j(0)P_j(x)|}\leq
\alpha+\delta
\end{equation*}
 for  any   arbitrary small $\delta$.
Thus,  equality \eqref{alpha} takes place.

\end{proof}

\begin{remark}
We should mention that the description  of the  Krein and  Friedrichs extensions
  given in Theorems~\ref{kre} and  ~\ref{T_F} in the scalar case coincides with one obtained earlier by Brown and Christiansen in ~\cite{BrChr05}.
\end{remark}
\begin{remark}
\end{remark}
\begin{itemize}
\item[\textit{(i)}]The  condition   $B_j>0$   can  be dropped. Indeed,  it is easy to show that
scalar   Jacobi  matrix with  arbitrary real $B_j$ is  unitarily equivalent   to  the
Jacobi  matrix  with positive $B_j$. The  unitary equivalence  is  established by the
diagonal  matrix with $1$ and $-1$ on the  diagonal. Besides, if  $B_j<0$,  then $1$ and
$-1$  have to stand next to each other   in  the  same  rows as $B_j$.

\item[\textit{(ii)}] The  fact that all zeroes of $P_j(\cdot)$  belong to $[0, \infty)$ might  be
also derived   from the holomorphicity  of  the Weyl
function,  corresponding to another boundary triplet  (see \cite[Proposition 10.1(2)]{DerMal95}),  on   $(-\infty, 0)$.

\item[\textit{(iii)}] In \cite{BrChr05}, authors obtained  the   description  \eqref{fr} by a
different method.
  Namely, they  essentially  used  the  fact that the  minimal (or principal) solution  $u=(u_j)$  of the equation
  $(\mathbf{J}u)_j=0,\,\, j\geq 1,$  has the  form $u=(u_j)=(P_j(0)-\alpha Q_j(0))$  and  belongs to  the   domain
  $\dom(T_F)$.


  \item[\textit{(iv)}] In \cite[Chapter 5, \S 3]{Kre73} (see also \cite{BrChr05}), it was  noted that  all  solutions $\widetilde{\mu}$ of  the Stieltjes moment problem
  associated  with $T_{\min}$ lie between the solutions $\mu_K$ and $\mu_F$  coming from the  Friedrichs and Krein
  extensions in the  following sense
\[\int_0^\infty\frac{d\mu_K(t)}{x-t}\leq\int_0^\infty\frac{d\widetilde{\mu}(t)}{x-t}\leq
\int_0^\infty\frac{d\mu_F(t)}{x-t},\quad x<0. \]
\end{itemize}
Proposition  \ref{prkf}$(v)$ leads to the  description of all non-negative
self-adjoint extensions of $T$ in the scalar case.
\begin{corollary}
Assume $p=1$. Let  $\Pi=\{\mathcal{H},\Gamma_0,\Gamma_1\}$  be the  boundary triplet for
$T_{\max}$ given by \eqref{tripl} and let $M(\cdot)$ be the corresponding Weyl function.
The set of all   non-negative self-adjoint extensions $T_h$ of the operator
$T_{\min}$ is parameterized as follows
\begin{equation}\label{param}
\dom(T_h)=\big\{u\in \dom(T_{\max}):(\Gamma_1-h\Gamma_0)u=0\big\},
\end{equation}
 $h\in[-\infty;\alpha]$ where $\alpha$ is defined by \eqref{a}. In particular,
  $$\dom(T_{-\infty})=\big\{u\in \dom(T_{\max}):\Gamma_0u=0\big\}.$$
\end{corollary}
\begin{proof}
First note  that for $h=\alpha$  and $h=-\infty$ the statement was  proved above.
Indeed,
$$\dom(T_\alpha)=\{u\in \dom(T_{\max}):(\Gamma_1-\alpha
\Gamma_0)u=0\}=\dom(T_F)$$  and
$$\dom(T_{-\infty})=\dom(T_{\max})\upharpoonright\ker(\Gamma_0)=\dom(T_K).$$ Thus, it remains to prove that for  $h<\alpha$  formula \eqref{param} defines non-negative self-adjoint extension. The result is implied by combining Proposition \ref{prkf}$(v)$ with Theorem \ref{T_F}.

Indeed, consider a new boundary triplet
$\widetilde{\Pi}=\{\widetilde{\mathcal{H}}, \widetilde{\Gamma_0},\widetilde{\Gamma_1}\}$
\begin{equation*}
\widetilde{\mathcal{H}}=\mathbb{C}, \qquad
\widetilde{\Gamma_0}f=\Gamma_1f-\alpha\Gamma_0f, \qquad
\widetilde{\Gamma_1}f=-\Gamma_0f,
\end{equation*}
where $\Gamma_0, \Gamma_1$ are  given by ~\eqref{tripl}. One easily  obtains that  the
corresponding Weyl  function is $\widetilde{M}(z)=(\alpha-M(z))^{-1}$. Taking into
account the above information about "limit values" of the Weyl function $M(\cdot)$, we
get that $\widetilde{M}(-\infty)$ is  "pure" linear relation, i.e.,
$\widetilde{M}(-\infty)=\{0,\cH\}$, and $\widetilde{M}(0)=0$. Hence, by Proposition
\ref{prkf}$(ii)$,
$T_0:=T_{\max}\upharpoonright\ker(\widetilde{\Gamma_0})=T_F$. Equation
\eqref{param}  in terms of the new  boundary triplet takes  the  form
$$\dom(T_h)=\Big\{u\in \dom(T_{\max}):\Big(\widetilde{\Gamma}_1-\frac
1{\alpha-h}\widetilde{\Gamma}_0\Big)u=0\Big\}.$$    Applying  Proposition
\ref{prkf}$(v)$, we get that $T_h\geq0$ if and only if
$\frac1{\alpha-h}>\widetilde{M}(0)=0$ or $h<\alpha.$
\end{proof}

\end{document}